\documentclass[11 pt,a4paper]{article}

\usepackage{amsmath,amsthm,amsfonts,amssymb,color}

\parskip=\medskipamount
\def\onehalf{{\textstyle\frac12}}
\def\oneqtr{{\textstyle\frac14}}
\def\onethird{{\textstyle\frac13}}

\def\R{{\mathbb R}}
\def\E{{\mathbb E}}

\def\hook{{\mathchoice{\vrule height 0pt depth 0.4pt width 3pt
\vrule height 5pt depth 0.4pt \kern 3pt} {\vrule height 0pt depth
0.4pt width 3pt \vrule height 5pt depth 0.4pt \kern 3pt} {\vrule
height 0pt depth 0.2pt width 1.5pt \vrule height 3pt depth 0.2pt
width 0.2pt \kern 1pt} {\vrule height 0pt depth 0.2pt width 1.5pt
\vrule height 3pt depth 0.2pt width 0.2pt \kern 1pt} }}

\def\d{\mbox{d}}
\def\tr{\text{tr}}
\def\det{\text{det}}
\def\diag{\text{diag}}
\def\dim{\text{dim}}
\def\pmr{\text{pm}}
\def\mtt#1{\mathtt{#1}}
\def\sym{\text{\tiny sym}}
\def\skew{\text{\tiny skew}}
\def\cS{\mathcal{S}}
\def\bA{\mathbf{A}}

\def\bQ{\mathbf{Q}}
\def\bb{\mathbf{b}}

\theoremstyle{plain}
\newtheorem{thm}{Theorem}[section]
\newtheorem{lem}[thm]{Lemma}
\newtheorem{propn}[thm]{Proposition}
\newtheorem{cor}[thm]{Corollary}

\theoremstyle{definition}
\newtheorem{defn}[thm]{Definition}

\newcommand{\pd}[2]{\frac{\partial#1}{\partial#2}}
\newcommand{\hnabla}{\hat{\nabla}}
\newcommand{\Sp}{\text{Sp}} 

\begin{document}
\begin{titlepage}

\title{Quadratic forms and the expansion and rotations of linear endomorphisms}

\author{Geoff Prince
\thanks{Email: {\tt g.prince\char64 latrobe.edu.au}}
\\Department of Mathematics and Statistics, La Trobe University,\\ Victoria 3086,
Australia.}
\date{July 4, 2023}

\maketitle

\begin{abstract}
 {\small New expansionary and rotational quadratic forms are constructed for $\E^n$-endomorphisms. Relations amongst the various eigenvalues, eigendirections and matrix invariants are established,
 including propositions on complexity and geometric multiplicity.
 The underlying construction involves a novel, almost-orthogonal expansion based on two-plane rotations. The development is strongly geometric in flavour and has application to the
 theory of connections, of which the Frenet case on $\E^3$ is given as a model.}
\end{abstract}

\footnotetext{ 2020 {\em  Mathematics Subject Classification.}
Primary 15A04,15A63,53A45  Secondary 15A15,15A18,53Z50}

\footnotetext{{\it Key words and phrases.} linear operators, matrix invariants, quadratic forms,
rotation and expansion of real endomorphisms, eigenspaces, eigenvalues, Frenet frame}
\end{titlepage}
\section*{Introduction}

In the presence of a basis on $\E^n$ there are two important additive decompositions of a real $n\times n$ matrix: one into symmetric and skew-symmetric parts
and the other being via the columns of the matrix itself, each holding the unique coordinates of the images of the basis vectors under the associated linear operator.
I will identify a third, unifying decomposition creating a fundamental role for quadratic forms representing the expansion and rotations of the underlying $\E^n$-endomorphism.
The strong geometric character of this construction allows, for example,  the identification of the eigenspaces of the endomorphism as the common zero subspaces of the rotations
and their dimension as the geometric multiplicity of the eigenvalues.
In addition, the Cayley-Hamilton theorem and the Newton trace formulae can be used to construct relations between the real eigenvalues of the quadratic forms and the possibly
complex ones of the endomorphism.

The main result of this paper is then theorem \ref{thm1}: \newline
{\em
Let $A$ be a non-trivial endomorphism on $\E^n$ with orthonormal basis $\mathbf{b}$ and non-zero $u\in\E^n.$
Then
\begin{equation}
A(u)=\bA^e(\hat u)u+\sum_{\stackrel{k,l}{k<l}}\bA^r_{kl}(\hat u)R_{kl}(u).\notag
\end{equation}}
In this expression $\bA^e$ and $\bA^r_{kl}$ are the above expansionary and rotational quadratic forms and the $R_{kl}$ are quasi-rotations of the basis two-planes $\Sp\{b_k,b_l\}.$
The $\onehalf n(n-1)$ terms in the sum on the right hand side are all (possibly trivially) orthogonal to $u$ although they are clearly linearly dependent when $n>2$.

A significant motivation has been the importance of tangent space endomorphisms in differential geometry, particularly the shape maps associated with metric and other connections
as well as the tangent map associated with smooth flows. The Frenet shape map is included as a model demonstrating the relationship between the rotational quadratic forms and the
conventional constructs of torsion and curvature for $\E^3$-flows. Because of this geometric emphasis, both on the construction and in the motivation, I have only included pointers
to the full complex case.

The paper begins with a review of the known results on $\E^2$ from \cite{GP2000} and some classical results on quadratic forms.
This is followed by the introduction of the quasi-rotations and the expansionary and rotational quadratic forms on $E^n$.
The main results are developed and are followed by an exploration of the Cayley-Hamilton and Newton trace formulae.
In the last section the constructions are illustrated on $\E^3$ and then using the Frenet shape map as a specific endomorphism and example of the application to the theory of connections on manifolds.
Some useful results and identities appear in an appendix.

{\em Author's note} This is an expository paper dealing with elementary concepts in linear algebra.
To the best of my knowledge this ground has not been covered before although it certainly would have been no surprise to find it in a nineteenth century treatise.
Nonetheless, I have found the development illuminating along with some insights in differential geometry; I hope it will find utility elsewhere.

\subsection*{Notation}
In what follows $A$ is a non-zero  endomorphism on $\E^n$ with matrix representation $[A]_\mathbf{b}$
relative to a basis $\mathbf{b}:=\{b_1,\dots,b_n\}$. $A^i_j$ is the element in the $i^\text{th}$ row and $j^\text{th}$ column of $[A]_\mathbf{b}$. The indices can be raised and lowered with the identity,
for example, $A_{kj}:=\sum_{i}A^i_j\delta_{ik}.$ $u \in \E^n$ is assumed non-zero with $\hat u:=u/\|u\|$. The coordinate column vector of $u$ relative to $\mathbf{b}$ is
$$
[u]_\mathbf{b}:= \left[\begin{matrix}u^1 \\ \vdots \\ u^n\end{matrix}\right]
$$
where $u=u^1b_1+\dots u^nb_n$.
The summation convention is used throughout so that, for example, $u=u^ib_i,$ and
$[A(u)]_\mathbf{b}=[A]_\mathbf{b}[u]_\mathbf{b}$ is equivalent to $A(u)^i=A^i_ju^j.$ Summation is not implied on like-positioned  indices, for example, $u^ju^j.$
he basis label is dropped when $\mathbf{b}=\mathbf{e},$ the natural basis. The standard inner product on $\E^n$ is denoted by $u\cdot v$ throughout.

Denoting by $A^\ast$ the adjoint of $A$ relative to the standard inner product, we define symmetric and skew operators, $A^{\sym},A^{\skew}$ in the usual way
\[
A^{\sym}:=\onehalf(A+A^\ast),\quad A^{\skew}:=\onehalf(A-A^\ast).
\]

Relative to an orthonormal basis $\mathbf{b}$ the corresponding matrix representations are, with superscript $T$ denoting matrix transpose,
\[
[A^{\sym}]_\mathbf{b}=\onehalf ([A]_\mathbf{b}+[A]_\mathbf{b}^T),\quad [A^{\skew}]_\mathbf{b}=\onehalf ([A]_\mathbf{b}-[A]_\mathbf{b}^T)
\]
so that
\begin{equation}
[A]_\mathbf{b}=[A^{\sym}]_\mathbf{b}+[A^{\skew}]_\mathbf{b}\label{sym/skew}.
\end{equation}
(While an orthonormal basis is not required for the decomposition itself the matrix properties of symmetry and skew-symmetry  are only  invariant under change of orthonormal basis.)
The symmetric and skew symmetric parts of $[A]_\mathbf{b}$ will also be denoted $[A]^{\sym}_\mathbf{b}$ and $[A]^{\skew}_\mathbf{b}.$

The determinant of $[A]_\mathbf{b}$ is denoted $\text{det}(A)$ and the trace by $\text{tr}(A),$ both being independent of $\bb.$
On occasion matrices will be used without regard to an underlying linear map in which case they will be denoted $\mathtt{B, D, S}$, etcetera.
When there is no ambiguity $\mtt{A}$ may be used instead of $[A]_\mathbf{b}.$

Finally, quadratic forms are written boldface, such as $\bQ,$ and the action of their symmetric matrix representations as self-adjoint linear operators are denoted $Q.$ So, for example,
$\bQ(u)=Q(u)\cdot u.$ In either case the matrix representation will be written as $[Q]$ rather than $[\bQ].$

\section{Planar endomorphisms}\label{Section 1}

See \cite{GP2000}. Let $A$ be a linear operator on $\mathbb{E}^2$, then for any
non-zero $u \in \mathbb{E}^2$
\[
 A(u) = (A(u)\cdot \hat u) \hat u + (A(u) \cdot \hat
u^\bot) \hat u^\bot,
\]
where $u^\bot:= R_{{\pi\over 2}} (u)$ (so
that if $u = (u^1, u^2)$ then $u^\bot = (-u^2, u^1))$.  Hence
\[A
(u) = (A (\hat u) \cdot \hat u) u + (A (\hat u) \cdot \hat u^\bot)
u^\bot
\]
and, since the adjoint of $R_{{\pi\over 2}}$ is
$R_{-{\pi \over 2}}$ relative to the standard inner product,
\[
A(u) = (A(\hat u) \cdot \hat u) u + (R_{-{\pi\over 2}}
 \circ A (\hat u)
\cdot \hat u) u^\bot
\]
which we rewrite as
\begin{equation}\label{E2 decomp}
A(u) = \bA^e(\hat u)u + \bA^r(\hat u)u^\bot,
\end{equation}
where $\bA^e$ and $\bA^r$ are the quadratic forms defined by
\[
 \bA^e
(\hat u) := A(\hat u) \cdot \hat u, \quad \bA^r
(\hat u) := R_{-{\pi \over2}} \circ A (\hat u) \cdot \hat u.
\]
If the matrix representation of $A$ (relative to the natural basis) is
\[
[A] = \left(\begin{matrix} A^1_1 &
A^1_2 \\ A^2_1 & A^2_2 \end{matrix} \right)
\]
then the matrix representations of the two quadratic forms, $\bA^e$ and
$\bA^r$, are
\[
[A^e] = \left(\begin{matrix} A^1_1 & {1 \over
2}(A^1_2 + A^2_1) \\ {1 \over 2}(A^1_2 + A^2_1) & A^2_2
\end{matrix} \right),\  [A^r] = \left(\begin{matrix} A^2_1
& {1 \over 2}(A^2_2 - A^1_1) \\ {1 \over 2}(A^2_2 - A^1_1) & -
A^1_2 \end{matrix} \right).
\]

\begin{defn}
$\bA^e (\hat u)$ and $\bA^r (\hat u)$ are respectively called the {\it expansion} and {\it rotation}
of the operator $A$ in the direction of $\hat u$. The {\it average values of the expansion and rotation} of $A, $ denoted $\overline{\bA^e},$ $\overline{\bA^r},$ are the average
values of $\bA^e (\hat u)$ and $\bA^r (\hat u)$ respectively. That is, $\overline{\bA^e}=\onehalf\tr(A^e),$ $\overline{\bA^r}=\onehalf\tr(A^r).$
\end{defn}

For non-zero $u\in\E^2$ the matrix representation of $A$ relative to the basis $\mathbf{b}:=\{\hat u,\hat u^\bot\}$ is
\begin{equation}\label{E2 mat rep}
[A]_\mathbf{b} = \left(\begin{matrix} \bA^e(\hat u) &
-\bA^r(\hat u^\bot) \\ \bA^r(\hat u) & \bA^e(\hat u^\bot) \end{matrix} \right).
\end{equation}
Using this representation (not presented in \cite{GP2000}) applied to an orthonormal eigenbasis of $[A^r]$ shows that when $\bA^r$ is indefinite
the eigenvalues of $A$ are real and are bounded above and below by those of $[A^e].$ The more general result is given later in theorem \ref{Bromwich 06}.

The following proposition can be found in  \cite{GP2000}.

\begin{propn}
\label{max/min2}
\begin{itemize}
\item[]
\item[](a) If the quadratic form $\bA^e$ is not zero, then its eigenvalues are the maximum and minimum values of the expansion of
the map $A$ and these are achieved in the corresponding eigendirections. The map A has constant expansion in all directions
if and only if $\bA^e$ has a repeated eigenvalue.

\item[](b) If the quadratic form $\bA^r$ is not zero, then its eigenvalues are the maximum and minimum values of the rotation
of the map $A$ and these are achieved in the corresponding eigendirections.  The map A has constant rotation in every direction
if and only if $\bA^r$ has a repeated eigenvalue.

\item[](c) The zeros of $\bA^r$ occur in the eigendirections of $A$.

\end{itemize}
\end{propn}

%

The following main theorem in \cite{GP2000} relates the eigenvalues of $A$ to those of
$\bA^e$ and $\bA^r$ and follows from the observations that
\begin{equation}\label{invariants n=2}
 \tr(A) = \tr(A^e)\ \text{and}\ \det(A)=\det(A^r) + {1\over 4} \tr(A^e)^2
\end{equation}
using various matrix representations of $A$, $\bA^e$ and $\bA^r.$
\begin{thm}
\label{evals} Let the (possibly repeated) eigenvalues of $A$ be $\lambda_1, \lambda_2$ and those of $\bA^e$ and
$\bA^r$ be $\lambda^e_1, \lambda^e_2$ and
$\lambda^r_1, \lambda^r_2$ respectively.  Then
\begin{equation}\label{quadratic eqn}
\lambda_{1,2} = {\lambda^e_1 + \lambda^e_2 \over 2} \pm\sqrt{-\lambda^r_1 \lambda^r_2}.
\end{equation}
\end{thm}

 This theorem shows that the average expansion of $A$ is the arithmetic mean
of the maximum and minimum expansions.  More strikingly it shows
that $A$ has complex eigenvalues precisely when the maximum and
minimum rotations have the same sign, so that the map rotates
every direction in the same sense.  This is intuitively satisfying
given the behaviour of pure planar rotations. The theorem
has a corollary which explains the geometric multiplicity of
repeated eigenvalues of the map $A$.

\begin{cor}\label{geom mult n=2}
If the linear operator $A$ has a repeated eigenvalue $\lambda$ its
geometric multiplicity is $2$ if and only if $\lambda^r_1 = 0 = \lambda^r_2$ (so that
$\bA^r$ is the zero quadratic form) and its
geometric multiplicity is $1$ if and only if only one of
$\lambda^r_1, \lambda^r_2$ is
zero.
\end{cor}

This corollary shows that a repeated eigenvalue occurs when a
map's rotation is zero in at least one direction (and the
eigenvalue achieves geometric multiplicity $2$ only when the map
itself is a pure expansion).

We can now classify the eigenvalues of linear operators on
$\mathbb{E}^2$ according to their corresponding rotations using
$\lambda^r_1 \leq \bA^r (\hat u)
\leq \lambda^r_2, \ \forall \hat u \in S^1$:

\begin{center}
\begin{tabular}{c | l}
 zeros of $\bA^r$ on $[0,\pi)$ \quad    & \quad Eigenvalues of
$A$  \\
\hline none  &\quad complex \\
 1  &\quad repeated, geometric multiplicity \ 1 \\
 2   &\quad real and distinct \\
 $\infty$  &\quad  repeated, geometric multiplicity \ 2 \\
\end{tabular}
\end{center}

\noindent Needless to say the eigenvectors of $A$ are the zeros of
$\bA^r$.

\section{Generalisation to $\E^n$}\label{Section 2}

\subsection{Preliminaries: Quadratic Forms}\label{Section 2.1 QF}

In order to generalise the first two parts of proposition \ref{max/min2} we need the following results (see in part Caratheodory \cite{Car65}):

\begin{propn}\label{Q forms2}
Let $\bQ =\diag(\lambda_1,\dots,\lambda_n)$ be a quadratic form on $\E^n$ diagonalised by the orthonormal eigenvector basis $\mathbf{b}:=\{b_1,\dots,b_n\}.$
And let $x^i$ be natural co-ordinates associated with $\mathbf{b}.$
Then
\begin{itemize}
\item[(a)] the values $\lambda_i$ of $\bQ$ on the unit $n$-sphere are attained at points $\pm b_i$
\item[(b)] if $\lambda_1\le\lambda_2\le\dots\le\lambda_n$ then $\lambda_1$ and $\lambda_n$ are the minimum and maximum values of $\bQ$ on $S^n$
\item[(c)] $\lambda_k$ is the minimum value of $\bQ$ on $\{p\in S^n:x^1(p)=0,\dots,x^{k-1}(p)=0\}$
\item[(d)] $\lambda_k$ is the maximum value of $\bQ$ on $\{p\in S^n:x^{k+1}(p)=0,\dots,x^n(p)=0\}.$
\end{itemize}
\end{propn}

\begin{defn}
The average value, $\bar \bQ,$ of a quadratic form  $\bQ$ on the unit $n$-sphere is defined to be
\[
\bar \bQ:=\frac{\int_{S^n}\bQ \d \sigma}{\int_{S^n} \d \sigma}
\]
where the denominator is the area of $S^n.$
\end{defn}
Using known results (see, for example, \cite{Foll01}) and the co-ordinates $x^j$ above, it is straightforward to show that
\begin{propn}
Let $\bQ =\diag(\lambda_1,\dots,\lambda_n)$ be a quadratic form on $\E^n$ diagonalised by the orthonormal eigenvector basis $\{b_1,\dots,b_n\}.$
Then the average value of $\bQ$ on the unit $n$-sphere is 
\[
\bar \bQ=\frac{\tr(Q)}{n}
\]
and this is achieved at each of the $2^n$ points
\[
\frac{1}{\sqrt{n}}(\pm b_1\dots\pm b_n).
\]
\end{propn}

Now we'll need to know when the zeros of a quadratic form generate a zero-valued vector subspace, particularly those {\em maximal} subspaces of largest dimension.

\begin{propn}\label{Q forms1}
Suppose that $\bQ$ is a quadratic form on $\E^n$ ($n>2$) and let $Q$ be a self-adjoint endomorphism with $\bQ(u)=u\cdot Q(u).$
Let $u,\ v$ be any zeros of $\bQ$ then 
\[
u+v\ \text{is a zero of}\ \bQ\ \text{if and only if}\ u\cdot Q(v)=0.
\]
(Equivalently, $W$ is a zero-valued subspace of $\bQ$ if and only if $Q(W)\subset W^\bot$.)
\end{propn}

\begin{proof}
Let $B$ be the polar form of $\bQ$, then
\begin{align*}
B(u,v)&=\onehalf(\bQ(u+v)-\bQ(u)-\bQ(v))\\
&=\onehalf ((u+v)\cdot Q(u+v)-u\cdot Q(u)-v\cdot Q(v))\\
&=\onehalf (v\cdot Q(u)+u\cdot Q(v)) =u\cdot Q(v).
\end{align*}
So if $u$ and $v$ are zeros of $\bQ$ then $u+v$ is a zero of $\bQ$ if and only if $B(u,v)=u\cdot Q(v)=0$.
\end{proof}

{\bf Notes} If $\text{ker}(Q)=\{0\}$ then $\bQ$ must be indefinite to have non-trivial zeros (see proposition \ref{Q forms2}).
If $\text{ker}(Q)\neq \{0\}$ then it is a zero-valued vector subspace of $\bQ$ and necessarily a subspace of any maximal zero-valued subspace of $\bQ.$

\subsection{Endomorphisms and their quadratic forms}\label{section 2.2 Endos}

Given an orthonormal basis $\mathbf{b}$, the conventional ``two-plane" rotations of $\E^n$, $\hat R_{kl},\ (k<l),$ given by
$$\hat{R}_{kl}(b_k)=b_l, \ \hat{R}_{kl}(b_l)=-b_k,\ \hat{R}_{kl}(b_m)=b_m, m\neq k,l,$$
do not satisfy $\hat R_{kl}(u)\cdot u=0$ for all $u$ unless $n=2$. In order to generalise the planar case we will instead compose $\hat R_{kl}$
with the corresponding two-plane projection giving quasi-rotations $R_{kl}$:
\[
R_{kl}(b_k)=b_l, \ R_{kl}(b_l)=-b_k,\ R_{kl}(b_m)=0, m\neq k,l.
\]
The matrix representation $[R_{kl}]_\bb$ has entries
\begin{equation}\label{identity 1}
{[R_{kl}]_\bb}^i_j=\delta^i_l\delta_{jk}-\delta^i_k\delta_{jl}.
\end{equation}

With this definition we have
\[
u^\bot_{kl}:=R_{kl}(u)=u^kb_l-u^lb_k.
\]

While $u\cdot u^\bot_{kl}=0,$ these quasi-rotations do not, of course, preserve length.

This quasi-rotation definition can be invariantly applied to any two-dimensional subspace of $\E^n$ without the need for a full basis.
Moreover, $R_{kl}$ is clearly independent of the orthonormal generators (of the same orientation) of the fixed two plane $S_{kl}:=\Sp\{b_k,b_l\}=\Sp\{b'_k,b'_l\}$ simply because $R_{kl} =R'_{kl}.$

From a structural perspective, we note that $\{[R_{kl}]: 1\le k <l\le n\}\subset M_n(\R)$ generates the subspace of skew-symmetric $n\times n$ real matrices;
also that the $[R_{kl}]$ generate the Lie algebra of the orthogonal group. Specifically, a skew symmetric matrix $\mtt{B}$ has the expansion
\begin{equation}
\mtt{B}=-\sum_{k<l}\mtt{B}^k_l[R_{kl}]\quad (\mtt{B}\ \text{skew}).\label{skew decomp}
\end{equation}

While the $R_{kl}$ depend on a particular basis for their definition and do not survive orthogonal transformations, we can use \eqref{skew decomp}
to relate the quasi-rotations $\tilde R_{pq}$ in basis $\mathbf{\tilde b}$ to $R_{kl}$ in basis $\mathbf{b}.$ If $P$ is the orthogonal transformation
from $\mathbf{b}$ to $\mathbf{\tilde b}$ so that $[\tilde R_{pq}]_\bb=P^T[\tilde R_{pq}]_{\tilde \bb}P$ then \eqref{skew decomp} implies
\begin{equation*}
\tilde R_{pq}=-\sum_{k<l}(P^T[\tilde R_{pq}]_{\tilde \bb}P)^k_lR_{kl}
\end{equation*}
and applying \eqref{identity 1} gives 
\begin{equation}\label{Rttn basis change}
\tilde R_{pq}=-\sum_{k<l}(P^l_pP^k_q-P^l_qP^k_p)R_{kl}.
\end{equation}

Now we move on to the use of these rotations in creating a novel and useful spanning set for $\E^n.$

\begin{lem}\label{lemma 1}
For any non-zero $v\in\E^n$ and orthonormal basis $\mathbf{b}$
\[
\E^n=\Sp\{v, R_{pq}(v);\ \forall p,q;\ p<q\}.
\]
\end{lem}
\begin{proof}
The statement is equivalent to the claim that (the orthogonal complement) $\Sp\{v\}^\bot=\Sp\{R_{pq}(v);\forall p,q; p<q\}.$
So suppose, without loss of generality, that $v$ has non-zero $b_1$ component. Then $\{R_{12}(v),\dots,R_{1n}(v)\}\subset \{R_{pq}(v);\forall p,q, p<q\}$ is
linearly independent and so forms a basis for the orthogonal complement.
\end{proof}

The spanning set $\{v,v^\bot_{kl}: k,l =1,\dots,n; k<l\}$ for $\|v\|=1$ has a remarkable property, akin to an orthonormal basis.

\begin{propn}\label{magic propn}
Let $u$ be an arbitrary non-zero element of $\E^n$ and $v\in\E^n$ be any unit length vector, then
\begin{equation}\label{magic formula}
u=(u\cdot v)v+\sum_{k<l}(u\cdot v^\bot_{kl})v^\bot_{kl}
\end{equation}
and so
\begin{equation}\label{magic formula(2)}
\|u\|^2=(u\cdot v)^2+\sum_{k<l}(u\cdot v^\bot_{kl})^2.
\end{equation}
\end{propn}
\begin{proof}
Take the inner product of the right hand side of equation \eqref{magic formula} with $b_p$:
\begin{align*}
&\left((u\cdot v)v+\sum_{k<l}(u\cdot v^\bot_{kl})v^\bot_{kl}\right )\cdot b_p\\
=\ &u^qv_qv^p+\sum_{k<l}(u\cdot (v^kb_l-v^lb_k))(v^kb_l-v^lb_k)\cdot b_p\\
=\ &u^qv_qv^p+\sum_{k<l}(u^lv^k-u^kv^l)(v^k\delta_{lp}-v^l\delta_{kp})\\
=\ &u^pv^pv^p+\sum_{q\neq p}u^qv^qv^p+\sum_{r\neq p}u^pv^rv^r-\sum_{s\neq p}u^sv^pv^s\\
=\ &u^pv^qv_q=u^p \ \text{since}\ \|v\|=1.
\end{align*}
(The identity \eqref{magic formula(2)} follows by taking the inner product with $u$.)
\end{proof}
{\bf Notes}\newline
In the case where $v=b_i$ expression \eqref{magic formula} becomes the usual orthonormal expansion of $u$.\newline
In the standard complex case the quasi-rotations become
\[
R_{pq}(u):=\bar{u}^pb_q-\bar{u}^qb_p.
\]

Now we can create expansive and rotational quadratic forms for an endomorphism.

Let $A$ be a non-trivial endomorphism on $\E^n$ with orthonormal basis $\mathbf{b},$ and $u\neq 0.$
\begin{defn} {\em Expansion of an endomorphism}\newline
The {\em expansion} of $A$ in the direction $u$ is the value $\bA^e(\hat u)$ of the quadratic form $\bA^e$ defined by
\begin{equation}\label{expn_defn}
\bA^e(u):=A(u)\cdot u=A_{ij}u^iu^j.
\end{equation}
\end{defn}
Notice that $[A^e]_\mathbf{b}=[A]^{\sym}_\mathbf{b}.$

\begin{defn}{\em Rotations of an endomorphism}\label{rttns}\newline
The {\em rotation} of the direction $\hat u$ by $A$ in the two plane $\Sp\{b_k,b_l\}$ is the value $\bA^r_{kl}(\hat u)$ of the quadratic form $\bA^r_{kl}$ defined by
\begin{equation}\label{rttn_defn}
\bA^r_{kl}(u):=A(u)\cdot R_{kl}(u)=(R^\ast_{kl}\circ A)(u)\cdot u=-(R_{kl}\circ A)(u)\cdot u,
\end{equation}
where $R^\ast_{kl}$ is the adjoint of $R_{kl}$ relative to the standard inner product.
\end{defn}
A short calculation gives
\[
\bA^r_{kl}(u)=A^l_ju^ju^k-A^k_ju^ju^l.
\]

Note that $\bA^r_{kl}(b_m)=0$ for $m\neq k, l$ and so
\begin{equation}\label{skew cmpts}
\tr(A^r_{kl})=\bA^r_{kl}(b_l)+\bA^r_{kl}(b_k) = -A^k_l+A^l_k=-2([A^{\skew}]_\mathbf{b})^k_l,
\end{equation}
and as a result \eqref{skew decomp} gives
\begin{equation}\label{skew decomp 2}
A^\skew=\onehalf\sum_{k<l}\tr(A^r_{kl})R_{kl}
\end{equation}
demonstrating how the components of $A^{\skew}$ are constructed from the (diagonal) components of $\bA^r_{kl}$, remembering that these
quadratic forms depend on $\mathbf{b}$ through the rotations $R_{kl}.$

We can relate rotational quadratic forms in different orthonormal frames
using \eqref{Rttn basis change}

\begin{equation}\label{A^r basis change}
{\tilde \bA}^r_{pq}(u):=A(u)\cdot\tilde R_{pq}(u)=-\sum_{k<l}(P^l_pP^k_q-P^l_qP^k_p)R_{kl}(u)\cdot A(u)=-\sum_{k<l}(P^l_pP^k_q-P^l_qP^k_p)\bA^r_{kl}(u).
\end{equation}

It is clear from \eqref{rttn_defn} that the common zeros of the $\bA^r_{kl}$ are exactly the eigenvectors of $A.$ As a result $\tilde \bA^r_{pq}$ and $\bA^r_{kl}$ share their common zeros.
The following proposition gives another shared property.

\begin{propn}
 If $\bA^r_{kl}$ and $\tilde \bA^r_{pq}$ are the rotational quadratic forms of an endomorphism $A$ on $\E^n$ relative to bases $\mathbf{b},\ \mathbf{\tilde b}$ respectively, then
\begin{equation}
\sum_{k<l}\tr(A^r_{kl})=\sum_{p<q}\tr(\tilde A^r_{pq}).
\end{equation}

\end{propn}

\begin{proof}
Let $R_{kl}$ and $\tilde R_{pq}$ be the two-plane rotations in bases $\mathbf{b},\ \mathbf{\tilde b}$ respectively.
Then, since the trace is independent of basis,
\[
\tr(\tilde A^r_{pq})=\sum_i\tilde \bA^r_{pq}(\tilde b_i)=\sum_i\tilde \bA^r_{pq}(b_i)=\sum_i\tilde R^T_{pq}\circ A(b_i)\cdot b_i=-\sum_{k<l}(P^T[\tilde R_{pq}]_\mathbf{\tilde b}P)^k_l\tr(A^r_{kl}),
\]
using $\tilde R_{pq}=-\sum_{k<l}(P^T[\tilde R_{pq}]_{\tilde
\bb}P)^k_lR_{kl}.$ Hence
\begin{align}
\sum_{p<q}\tr(\tilde  A^r_{pq})&=-\sum_{p<q}\sum_{k<l}(P^T[\tilde R_{pq}]_\mathbf{\tilde b}P)^k_l\tr(A^r_{kl}))\\
&=-\sum_{k<l}(P^T(\sum_{p<q}[\tilde R_{pq}]_\mathbf{\tilde b})P)^k_l\tr(A^r_{kl}).
\end{align}
Observing that the matrix $\sum_{p<q}[\tilde R_{pq}]_\mathbf{\tilde b}$ is skew with every upper triangular entry equal to $-1$ gives the result.
\end{proof}
This is equivalent to showing that the sum of the entries of $[A^\skew]$ is invariant under orthogonal transformations.

Section \ref{Section 2.1} explores the expansionary and rotational forms in privileged bases associated with the decomposition $A=A^{\sym}+A^{\skew}.$
However, at this stage we can draw upon the real case of Bromwich's 1906 theorem \cite{Brom06} (see \cite{Mirsky90} p.389 for more detail):
\begin{thm}\label{Bromwich 06}
Let $A$ be an endomorphism on $\E^n$ with associated endomorphisms $A^{\sym}, A^{\skew}.$ If $\lambda$ is any eigenvalue of $A$ then
\[
\nu\leq \Re(\lambda)\leq N \ \text{and}\ \mu\leq\Im(\lambda)\leq M
\]
where $\nu, N$ are the minimum and maximum eigenvalues of $A^{\sym}=A^e$ and $\mu, M$ are the minimum and maximum eigenvalues of the Hermitian matrix $\frac{1}{i}[A^{\skew}]=\frac{1}{2i}\sum_{k<l}\tr(A^r_{kl})[R_{kl}].$
\end{thm}
Aside: Bromwich's theorem is easy to prove in the quadratic form framework because (in the complex scenario) both the corresponding Hermitian
quadratic forms achieve each of the eigenvalues of $A$ in the corresponding eigendirections, obviously between their respective extreme values.

\begin{propn}{Generalisation of proposition \ref{max/min2}}
\label{max/min3}\newline
Let $A$ be a non-trivial endomorphism on $\E^n$, $\mathbf{b}$ an orthonormal basis and non-zero $u\in\E^n.$\\ Then 
\begin{itemize}
\item[](a) If the quadratic form $\bA^e$ is not zero, then its maximum and minimum eigenvalues are the maximum and minimum values respectively of the expansion of
the map $A$ and these are achieved in the corresponding eigendirections. The real parts of the eigenvalues of $A$ lie between the maximum and minimum expansions.
The map A has non-zero constant expansion on a subspace $W$ with $\dim(W)>1$ if and only if $W$ is the eigenspace of a repeated eigenvalue of $\bA^e.$

\item[](b) If the quadratic form $\bA^r_{kl}$ is not zero, then its maximum and minimum eigenvalues are the maximum and minimum values respectively of the rotation by $A$ in
 the two-plane $\Sp\{b_k,b_l\}$ and these are achieved in the corresponding eigendirections.  The map A has non-zero constant rotation in $\Sp\{b_k,b_l\}$ on a subspace $W$ with $\dim(W)>1$
if and only if $W$ is an invariant subspace of  a repeated eigenvalue of $\bA^r_{kl}.$

\item[](c) $u$ is an eigenvector of $A$ $\iff$ $u$ is a common zero of all $\bA^r_{kl}.$

\end{itemize}
\end{propn}

Noting that, as a result of part (c) of proposition \ref{max/min3} and the linear independence of eigenvectors of distinct eigenvalues, the maximal
common zero-valued subspaces of the $\bA^r_{kl}$ are disjoint,
the matter of geometric multiplicity resolves immediately, generalising corollary \ref{geom mult n=2}.

\begin{cor}\label{geom mult arb n}
Let $A$ be a non-trivial endomorphism on $\E^n$ with orthonormal basis $\mathbf{b}$. Then
\begin{itemize}
\item[(a)] For odd $n$ the $\bA^r_{kl}$ have at least one common zero.
\item[(b)] The dimension of each (maximal) common zero-valued subspace
of the $\bA^r_{kl}$ is equal to the geometric multiplicity of the corresponding eigenvalue of $A$.
\end{itemize}
\end{cor}
It goes without saying that if any of the $\bA^r_{kl}$ are positive or negative definite then $A$ has no real eigenvalues, and, for $n$ odd, all of the rotational forms are indefinite.

Motivated by the $n=2$ decomposition \eqref{E2 decomp}, lemma \ref{lemma 1} and proposition \ref{magic propn}, I now present the key to the generalisation of the $n=2$ results to arbitrary dimension.

\begin{thm}{Main theorem}\label{thm1}

Let $A$ be a non-trivial endomorphism on $\E^n$ with orthonormal basis $\mathbf{b}$ and non-zero $u\in\E^n.$
Then
\begin{equation}
A(u)=\bA^e(\hat u)u+\sum_{\stackrel{k,l}{k<l}}\bA^r_{kl}(\hat u)R_{kl}(u).\label{main result}
\end{equation}

\end{thm}
\begin{proof}
\ \newline

Unusually I offer two proofs. From proposition \ref{magic propn}

\begin{align*}
A(u)&=(A(u)\cdot\hat u)\hat u + \sum_{\stackrel{k,l}{k<l}}(A(u)\cdot {\hat u}^\bot_{kl}){\hat u}^\bot_{kl}\\
&=(A(\hat u)\cdot \hat u)u + \sum_{\stackrel{k,l}{k<l}}(A(\hat u)\cdot {\hat u}^\bot_{kl}){u}^\bot_{kl}\\
&=\bA^e(\hat u)u+\sum_{\stackrel{k,l}{k<l}}\bA^r_{kl}(\hat u)R_{kl}(u).
\end{align*}
This ends the first proof. If this seems a manifest abuse of the linear expression \eqref{magic formula}, instead set
\[
w_u:=A(u)-\bA^e(\hat u)u-\sum_{\stackrel{k,l}{k<l}}\bA^r_{kl}(\hat u)R_{kl}(u).
\]
We will show for all non-zero $u$ that $w_u$ is orthogonal to the elements of
the spanning set in lemma \ref{lemma 1} and is thus zero. Clearly $w_u\cdot
u=A(u)\cdot u-\bA^e(\hat u)u\cdot u=0,$ so it remains to demonstrate this for
$R_{pq}(u),$ assumed non-zero. Now
\[
R_{kl}(u)\cdot R_{pq}(u)=(u^kb_l-u^lb_k)\cdot(u^pb_q-u^qb_p)=u^ku^p\delta_{lq}-u^ku^q\delta_{lp}-u^lu^p\delta_{kq}+u^lu^q\delta_{kp}
\]
and
\[
\bA^r_{kl}(\hat u):=A(\hat u)\cdot R_{kl}(\hat u)=A^i_j\hat u^jb_i\cdot (\hat u^kb_l-\hat u^lb_k)=A_{lj}\hat u^j\hat u^k-A_{kj}\hat u^j\hat u^l.
\]
Hence
\[
\sum_{\stackrel{k,l}{k<l}}\bA^r_{kl}(\hat u)R_{kl}(u)\cdot R_{pq}(u)=\sum_{\stackrel{k,l}{k<l}}\left(A_{lj}\hat u^j\hat u^k-A_{kj}\hat u^j\hat u^l\right)
\left(u^ku^p\delta_{lq}-u^ku^q\delta_{lp}-u^lu^p\delta_{kq}+u^lu^q\delta_{kp}\right).
\]
The terms in this sum are symmetric $k$ and $l$ (the two factors being skew symmetric) and so the sum on $k,l$ with $k<l$ can be replaced by a sum on all $k,l$ without the $k=l$ terms:
\begin{align*}
\sum_{\stackrel{k,l}{k<l}}\bA^r_{kl}(\hat u)R_{kl}(u)\cdot R_{pq}(u)&=\frac{1}{2}\sum_{\stackrel{k,l}{k\neq l}}\left(A_{lj}\hat u^j\hat u^k-A_{kj}\hat u^j\hat u^l\right)\left(u^ku^p\delta_{lq}-u^ku^q\delta_{lp}-u^lu^p\delta_{kq}+u^lu^q\delta_{kp}\right)\\
&-\sum_{k=1}^{n}\left(A_{kj}\hat u^j\hat u^k-A_{kj}\hat u^j\hat u^k\right)\left(u^ku^p\delta_{kq}-u^ku^q\delta_{kp}-u^ku^p\delta_{kq}+u^ku^q\delta_{kp}\right),
\end{align*}
each term in the second sum is clearly zero. This leaves
\[
\sum_{\stackrel{k,l}{k<l}}\bA^r_{kl}(\hat u)R_{kl}(u)\cdot R_{pq}(u)=\frac{1}{2}\sum_{\stackrel{k,l}{k\neq l}}\left(A_{lj}\hat u^j\hat u^k-A_{kj}\hat u^j\hat u^l\right)\left(u^ku^p\delta_{lq}-u^ku^q\delta_{lp}-u^lu^p\delta_{kq}+u^lu^q\delta_{kp}\right).
\]
Simplification gives
\[
\sum_{\stackrel{k,l}{k<l}}\bA^r_{kl}(\hat u)R_{kl}(u)\cdot R_{pq}(u)=\sum_{k}\left(A_{qj}\hat u^j\hat u^p-A_{pj}\hat u^j\hat u^q\right)u^ku^k=A_{qj} u^j u^p-A_{pj}u^j u^q=\bA^r_{pq}(u)
\]
and so
\[
w_u\cdot R_{pq}(u)=A(u)\cdot R_{pq}(u)-\bA^r_{pq}(u)=0
\]
as required.

\end{proof}
Taking the inner product of \eqref{main result} with $A(\hat u)$ and using the definitions of the quadratic forms gives
\begin{cor}\label{main thm cor}
\begin{equation}
\|A(\hat u)\|^2=\mathbf{A}^e(\hat u)^2+\sum_{k<l}{\mathbf{A}^r_{kl}(\hat u)}^2.\label{magic formula(3)}
\end{equation}
\end{cor}
The next corollary follows simply by applying the result of theorem \ref{thm1} to the adjoint, $A^\ast,$ of $A.$ It identifies the quadratic forms associated
with the symmetric and skew symmetric parts of $A$.
\begin{cor}\label{cor comms}
\begin{align}
A^{\sym}(u)&=\bA^e(\hat u)u+\onehalf\sum_{\stackrel{k,l}{k<l}}\left([A,R_{kl}](\hat u)\cdot\hat u\right) R_{kl}(u)\\
A^{\skew}(u)&=-\onehalf\sum_{\stackrel{k,l}{k<l}}\left(\{A,R_{kl}\}(\hat u)\cdot \hat u\right) R_{kl}(u),
\end{align}
where $[\ ,\ ]$ and $\{\ ,\ \}$ are the usual commutator and anti-commutator operations.
\end{cor}
For comparison, and because $T(u)\cdot u=\onehalf(T+T^\ast)(u)\cdot u$,
\begin{cor}\label{cor comms II}
\begin{align}
\bA^e(\hat u)&=A^{\sym}(\hat u)\cdot\hat u\\
\bA^r_{kl}(\hat u)&=\onehalf[A^{\sym},R_{kl}](\hat u)\cdot \hat u-\onehalf\{A^{\skew},R_{kl}\}(\hat u)\cdot \hat u.
\end{align}
\end{cor}

It's interesting to compare the decomposition \eqref{main result}  with the standard Euler--Cauchy--Stokes decomposition theorem~\cite{T77},
which is itself an extension of the symmetric/skew-symmetric decomposition:

\begin{thm} ( Euler--Cauchy--Stokes Decomposition Theorem)\ 
The matrix representation of $A$ can be uniquely decomposed as
\[
[A] = {\Theta \over 2} I_n + \Sigma + \Omega
\]
 where $\Theta :=\tr(A)$ (so that $\displaystyle{\Theta \over n}$ is the
average expansion of $A$), \ $\Sigma$ is a trace-free symmetric matrix representing shear and $\Omega$ is a skew-symmetric matrix
representing the twist of the map.
\end{thm}

We close with the observation that, for arbitrary orthonormal basis $\bb$ and arbitrary operator $A$  (that is, no algebraic relations amongst the entries of $[A]_\bb$),
the set $\{[A^e]_\bb,[A^r_{kl}]_\bb:k<l\}$ is linearly independent in the space of $n \times n$ real matrices.

\subsection{Matrix representations and invariants}\label{Section 2.1}

The goal of this section is, for arbitrary $n$, to completely describe the invariants of $A$ in terms of those of its expansionary and rotational quadratic forms.
We could try and produce a formula akin to \eqref{quadratic eqn} or, less ambitiously, like \eqref{invariants n=2}. If we choose the latter then we
should also account for invariants such as $\tr(AA^*)$ 
which can't be written in terms of $\tr(A^2)$ and $\tr(A)^2$. Key tools here will be the Cayley-Hamilton theorem and Newton's trace formulae. 

\subsubsection{Canonical bases}

We begin with a discussion of orthonormal bases distinguished by the quadratic forms.

The first observation about the result \eqref{main result} in comparison with \eqref{E2 decomp} is that the former does not immediately generate a matrix representation for $A$ akin to
\eqref{E2 mat rep}. This is because $\{\hat u, R_{kl}(\hat u): k<l\}$ is a linearly dependent set unlike $\{\hat u,\hat u^\bot\}.$
There are two obvious ways to deal with this. The first is to choose a linearly independent subset of these generators. For example, for a fixed $u$ the basis $\mathbf{b}$
can be re-ordered so that $u^1\neq 0,$ then the set
$\{\hat u, R_{1,l}(\hat u): l=2,\dots,n\}$ is linearly independent (but not orthonormal) and  a matrix representation can be produced.
This may be useful when we have specific values for $u$ and $\mathbf{b}$ in mind, perhaps in numerical applications.
The second approach avoids the issue and focuses on choices of $\mathbf{b}$ as follows.

The orthonormal basis $\mathbf{b}$ in the expression \eqref{main result} is arbitrary. However, the definition of $R_{kl}$ does not survive a change of basis.
But there are canonical bases associated with $A.$

Firstly, suppose that $[A]^{\skew}_\mathbf{b}\neq 0$ in \eqref{sym/skew}, then $[A]^{\skew}_\mathbf{b}$ is orthogonally similar
to the real block form
\begin{equation}
\text{diag}(\mathtt{A}^s_1,\mathtt{A}^s_3,\dots,\mathtt{A}^s_{2p+1},0,\dots,0)\ \text{with}\
\mathtt{A}^s_k:=\left[ \begin{matrix}0 & \lambda^s_k\\-\lambda^s_k&0 \end{matrix}\right]\ (k\ \text{odd})\label{skew block form1}
\end{equation}
and the non-zero, necessarily pure imaginary,  eigenvalues of $[A]^{\skew}_\mathbf{b}$ are $\pm \lambda^s_1i,\dots,\pm\lambda^s_{2p+1}i$ (\cite{Gant59}).
In this orthonormal basis
\begin{equation}
[A]^{\skew}_\mathbf{b}=-\lambda^s_1[R_{12}]-\dots -\lambda^s_{2p+1}[R_{(2p+1)(2p+2)}].\label{skew block form2}
\end{equation}
Notice from the identity \eqref{skew cmpts} that in this basis
\[
\lambda^s_k=-\frac{1}{2}\tr(A^r_{k(k+1)}),\ k=1,3,\dots,2p+1,
\]
which can be used in theorem \ref{Bromwich 06} to give bounds on the imaginary parts of the eigenvalues of $A$.
In this basis $[A]_\mathbf{b}$ is symmetric except on the non-zero diagonal $2\times 2$ blocks:
\[
A^i_j=A^j_i,\ i<j\ \text{and}\ (i,j)\neq (2\ell+1,2(\ell+1)).
\]
Application of corollary \ref{cor comms II} to give $\bA^r_{kl}$ doesn't appear to be productive.

In the case where additionally $\bA^e$ is zero, $[A]$ is skew-symmetric in any orthonormal basis $\bb$. In such a basis $[A^2]$ is non-zero and symmetric.
The relationship between the eigenspaces of $A^2$ and the invariant subspaces of $A$ is given by 
\begin{propn}
Suppose $A$ has skew-symmetric matrix representation $[A]_\mathbf{b}$ relative to an orthonormal basis $\mathbf{b}$.
Then the eigenspaces of $A^2$ are invariant under $A$ and the invariant subspaces of $A$ are invariant under $A^2.$

\end{propn}

So in the case that $[A]_\mathbf{b}$ is skew-symmetric we could change to any orthonormal basis of $A^2$ as a canonical basis in expression \eqref{main result}.
This option won't be pursued further here; in general we assume that $[A]_\mathbf{b}$ is neither symmetric nor skew-symmetric in the
arbitrary orthonormal basis $\mathbf{b}.$ We will work with bases diagonalising  $A^{\sym}$ next; such bases will generally be more useful than
those diagonalising $A^{\skew}.$ But note that  non-zero $A^{\skew}$ and $A^{\sym}$ commute if and only if $A$ and $A^\ast$ commute, that is when $A$ is normal; this will be dealt with shortly.

Suppose now that the quadratic form  $\bA^e$ is non-zero (equivalently $[A]^{\sym}_\mathbf{b}\neq 0$ in \eqref{sym/skew}); this is true independent of the choice of basis.
Since $\bA^e$ has a symmetric matrix representation, its eigenvalues are real and its eigenspaces provide possibly non-unique, but canonical, orthonormal bases for $\E^n$.
Any of these bases can be used for $\mathbf{b}$, giving the rotations $R_{kl}$ intrinsic meaning.

So assume that  $[A^e]_\mathbf{b}=\text{diag}(\lambda^e_1,\dots,\lambda^e_n)$ where $\lambda^e_k$ are the real eigenvalues of $\bA^e.$
$[A]_\mathbf{b}$ splits into an expansionary diagonal (symmetric) part and a rotational skew-symmetric part (denoted $D$ and $S$ respectively), seen by
putting $u=b_p$ in \eqref{main result} and taking the inner product with $b_q$ gives the matrix entries $A^q_p:=A(b_p)\cdot b_q$:
\begin{align}
A^p_p&=\bA^e(b_p)=\lambda^e_p\quad  (\text{no sum}),\label{MC1}\\
A^q_p&=-A^p_q=\bA^r_{pq}(b_p)=\bA^r_{pq}(b_q)=\onehalf\tr(A^r_{pq}),\quad p<q.\label{MC2}
\end{align}

Aside: It follows from these expressions that the elementary symmetric
polynomial invariants of $A$ can be expressed as composites of the
$\lambda^e_p$ and the traces of the $\bA^r_{pq}$, with the traces of the
rotations appearing, rather than the eigenvalues themselves, because of their
basis dependence. \newline Applying corollary \ref{cor comms II}, the
rotational forms in this basis are
\begin{align*}
\bA^r_{kl}(\hat u)&={\bA^{\sym}}^r_{kl}(\hat u)+{\bA^{\skew}}^r_{kl}(\hat u)\\
&=\onehalf[A^{\sym},R_{kl}](\hat u)\cdot \hat u-\onehalf\{A^{\skew},R_{kl}\}(\hat u)\cdot \hat u\\
&=\onehalf\left[D,[R_{kl}]\right](\hat u)\cdot \hat u-\onehalf\{S,[R_{kl}]\}(\hat u)\cdot \hat u,
\end{align*}
abbreviating $[R_{kl}]_\mathbf{b}$ to $[R_{kl}]$ and identifying $n\times 1$
matrices with $\E^n$. The matrix representations of the two rotational forms
have entries
\begin{align*}
[{A^{\sym}}^r_{kl}]^i_j:=&{A^{\sym}}^r_{kl}(b_j)\cdot b_i=\onehalf\left( -[R_{kl}]^i_mD^m_j+D^i_m[R_{kl}]^m_j\right)\\
=&\onehalf[R_{kl}]^i_j\left(D^i_i-D^j_j \right)=\onehalf(\delta_{jk}\delta_{il}-\delta_{jl}\delta_{ik})(\lambda^e_i- \lambda^e_j)
\end{align*}
and
\begin{align*}
[{A^{\skew}}^r_{kl}]^i_j:=&{A^{\skew}}^r_{kl}(b_j)\cdot b_i=-\onehalf\left( [R_{kl}]^i_mS^m_j+S^i_m[R_{kl}]^m_j\right)\\
=&\onehalf\left(\left(S^l_j\delta_{ki}-S^k_j\delta_{li}\right)+\left(S^l_i\delta_{kj}-S^k_i\delta_{lj}\right)\right).
\end{align*}

Finally, in this basis
\[
[A^r_{kl}]^i_j=\onehalf(\delta_{jk}\delta_{il}-\delta_{jl}\delta_{ik})(\lambda^e_i-\lambda^e_j)+\onehalf\left(\left(S^l_j\delta_{ki}-S^k_j\delta_{li}\right)+\left(S^l_i\delta_{kj}-S^k_i\delta_{lj}\right)\right)
\]
so that only the $k$ and $l$ rows and columns have nonzero entries. Notice that ${A^{\sym}}^r_{kl}$ makes no contribution to $\tr(A^r_{kl}).$

To end this section we use this basis to produce a characterisation of normal matrices in terms of their expansionary and rotational quadratic forms.
As earlier remarked non-zero $A^{\skew}$ and $A^{\sym}$ are simultaneously diagonalisable in this case. Also note that if either $[A^e]_\mathbf{b}=D=0$ or $S=0$ then $A$ is immediately normal.

\begin{propn}{Normal operators I}\label{normal I}\\
Let $A$ be a non-trivial endomorphism on $\E^n.$ Suppose that $\bA^e \neq 0$ so that $[A]_\mathbf{b}=D+S$ in an orthonormal eigenbasis $\mathbf b$
for $\bA^e$ with $[A^e]_\mathbf{b}=D=\text{diag}(\lambda^e_1,\dots,\lambda^e_n)$ and $S$ is skew. $\bA^r_{ij}$ are the rotational forms relative to this basis. Then
\begin{itemize}
\item[] $A$ is normal if and only if $DS=SD$
\item[] $A$ is normal if and only if $\tr(A^r_{ij})=0$ for each $i<j$ such that $\lambda^e_i\neq\lambda^e_j$

\item[] $A$ cannot be normal if the $\lambda^e_i$ are all distinct and at
    least one $\tr(A^r_{ij})\neq 0$
\item[] If all the $\lambda^e_i$ are distinct then $A$ is normal if and only if  $A=A^e$
\item[] If $A$ is normal but not symmetric then $\bA^e$ has at least one repeated eigenvalue.
\end{itemize}
\end{propn}

\begin{proof}
$[A]_\mathbf{b}[A]^T_\mathbf{b}=(D+S)(D-S)$ and $[A]^T_\mathbf{b}[A]_\mathbf{b}=(D-S)(D+S)$ so that $AA^*=A^*A \iff DS=SD.$
(It is a standard result that in general $A$ is normal if and only if its skew and symmetric parts commute.)

Now
\begin{equation}
(DS)^i_j-(SD)^i_j =D^i_kS^k_j-S^i_kD^k_j=D^i_iS^i_j-S^i_jD^j_j
\end{equation}
so that
\begin{equation}
DS=SD \iff S^i_j(D^i_i-D^j_j)=0\ \text{for all}\ i,j.
\end{equation}
Recalling that $\lambda^e_i=D^i_i$ and $S^i_j=-\onehalf\tr(A^r_{ij})$ establishes the remaining claims.
\end{proof}

The following results will be important in the next section.

\begin{propn}{Normal operators II}\label{normal II}\\
Suppose that $A$ is a normal operator on $\E^n$ and that in some orthonormal basis $[A]_\bb=D+S$, where $D$ is diagonal and non-zero $S$ is skew symmetric. Then
\begin{itemize}
\item[(a)] For positive integers $p,q$ the symmetric matrices $[(A^p)^e]_\bb$ and $[(A^q)^e]_\bb$ commute.
\item[(b)] There exists a basis $\bb$ in which $S^2$ is diagonal and which simultaneously diagonalises $[(A^p)^e]_\bb$ for all positive integers $p.$

\end{itemize}
\end{propn}

\begin{proof}\

\begin{itemize}
\item[(a)]
Assume that neither $[(A^p)^e]_\bb$ or $[(A^q)^e]_\bb$ are zero. Now
\begin{align*}
[(A^p)^e]_\bb[(A^q)^e]_\bb &=\frac{1}{4}([A^p]_\bb+[A^p]_\bb^T)([A^q]_\bb+[A^q]_\bb^T)\\
&=\frac{1}{4}([A^{p+q}]_\bb+[A^{p+q}]_\bb^T+[A^p]_\bb[A^q]_\bb^T+[A^p]_\bb^T[A^q]_\bb).
\end{align*}
Repeated use of the normality condition on the last two terms moves the $[A^q]_\bb^T$ and $[A^q]_\bb$ terms to the left, giving the desired result.

\item[(b)] The result in part (a) guarantees the existence of a basis $\bb$ simultaneously diagonalising the matrices $[(A^p)^e]_\bb$.
In such a basis
\[
[(A^p)^e]_\bb=\onehalf\left((D+S)^p+((D+S)^p)^T\right).
\]
These are polynomials in $D$ and $S$ and, because of normality ($DS=SD$), they commute with $S^2$ for all $p.$
This establishes the existence of the required basis.
\end{itemize}
\end{proof}

We next address the relationships between the matrix invariants of $A$ and those of the geometric quadratic forms, extending the results from the planar case in section \ref{Section 1}.

\subsubsection{The Cayley-Hamilton Theorem and Newton's Trace Formulae}\label{Section 2.1.1}

The aim here is to separately  use the Cayley-Hamilton theorem and Newton's trace identities (see eg \cite{Hohn64, Kal00}) to obtain
expressions for the matrix invariants of $A$ in terms of those of its geometric quadratic forms akin to the $n=2$ formula \eqref{invariants n=2},
although for obvious reasons an arbitrary $n$ version of the formula \eqref{quadratic eqn} won't be possible. Here the term `matrix invariants'
refers to the elementary symmetric polynomials in the entries of the matrix, rather than in its eigenvalues. Success is only partial and neither
of the two approaches deliver closed form expressions for the  matrix invariants of $A$ exclusively in terms of those of the expansionary and rotational forms.
Nonetheless the results in one case are definitive.

Suppose that $\mtt{B}$ is a real $n\times n$ matrix with characteristic polynomial
\begin{equation}
P(\mtt{B}):=\mtt{B}^n-\text{tr}(\mtt{B})\mtt{B}^{n-1}+\text{pm}^2(\mtt{B})\mtt{B}^{n-2}\dots+(-1)^n\text{det}(\mtt{B})I_n\label{ch_poly}
\end{equation}
where $\text{pm}^a(\mtt{B})$ is the sum of the principal minors of $\mtt{B}$ of degree (or order)  $a$.\newline
If we set $\text{pm}^1(\mtt{B}):=\text{tr}(\mtt{B})$, $\text{pm}^n(\mtt{B}):=\text{det}(\mtt{B})$ and $\text{pm}^0(\mtt{B}):=1$,
these being the elementary symmetric polynomials in the matrix entries, then
\begin{equation}
P(\mtt{B}):=\mtt{B}^n -\text{pm}^1(\mtt{B})\mtt{B}^{n-1}+\dots+(-1)^k\text{pm}^k(\mtt{B})\mtt{B}^{n-k}+\dots+(-1)^n\text{pm}^n(\mtt{B})I_n.\label{ch_poly2}
\end{equation}
For $1\leq k\leq n$ Newton's trace formulae are the identities 
\begin{equation}
\text{tr}(\mtt{B}^k-\text{pm}^1(\mtt{B})\mtt{B}^{k-1}+\dots+(-1)^{k-1}\text{pm}^{k-1}(\mtt{B})\mtt{B}+(-1)^k\text{pm}^k(\mtt{B})I_n)=(n-k)(-1)^k\text{pm}^k(\mtt{B}).\label{Newton}
\end{equation}
These lead to the recurrence formulae for the sequence of $\text{pm}^k(\mtt{B})$ in terms of the traces of powers of $\mtt{B}.$
In the complex case this yields the usual expressions for the matrix invariants in terms of the eigenvalues. We will return to the
Newton trace formulae after exploring the use of the Cayley-Hamilton theorem in producing higher order analogues of \eqref{invariants n=2}.

Using the Cayley-Hamilton theorem on the characteristic polynomial \eqref{ch_poly2} for our endomorphism $A$ and with unit $u,v\in\E^n$ gives
\begin{equation}
A^n(u)\cdot v -\text{pm}^1(A)A^{n-1}(u)\cdot v+\dots+(-1)^{n-1}\text{pm}^{n-1}(A)A(u)\cdot v+(-1)^n\text{pm}^n(A)u\cdot v=0.\label{CH1}
\end{equation}
Notice that if $v=u$ and  $pm^n(A)\neq 0$ then the last term is independent of $u$ so  the remainder has the same value for all $u$.
And setting $u=v=b_i$ and summing $i=1\ \text{to}\ n$ recovers the Newton trace formula for $k=n.$
But, more importantly for our context, using $v=u$ and $v=R_{kl}(u)$ respectively in \eqref{CH1}, gives
\begin{propn}
\begin{equation}
(\bA^n)^e(u) -\text{pm}^1(A)(\bA^{n-1})^e(u)+\dots+(-1)^{n-1}\text{pm}^{n-1}(A)\bA^e(u)+(-1)^n\text{pm}^n(A)=0, \label{CH2}
\end{equation}
and
\begin{equation}
(\bA^n)^r_{kl}(u) -\text{pm}^1(A)(\bA^{n-1})^r_{kl}(u) +\dots +(-1)^{n-1}\text{pm}^{n-1}(A)\bA^r_{kl}(u)=0.\label{CH4}
\end{equation}
\end{propn}

In \eqref{CH2} and \eqref{CH4} setting $u:=b_i$ and summing over $i$ yields trace formulae for the expansionary and rotational quadratic forms respectively (in a slightly simplified notation):
\begin{cor}
\begin{align}
&\ \tr(\bA^{n^e}) -\text{pm}^1(A)\tr({\bA^{n-1}}^e)+\dots+(-1)^{n-1}\text{pm}^{n-1}(A)\tr(\bA^e)+(-1)^n(n)\text{pm}^n(A)=0,\label{tr_CH2}\\
&\ \tr(\bA^{n^r}_{kl}) -\text{pm}^1(A)\tr({\bA^{n-1}}^r_{kl}) +\dots +(-1)^{n-1}\text{pm}^{n-1}(A)\tr(\bA^r_{kl})=0\label{tr_CH4}
\end{align}
(the first of these is just the Newton trace formula for $A^n).$
\end{cor}

Equations \eqref{tr_CH2},\eqref{tr_CH4} are candidates for linear equations for the invariants of $A$ in terms of the real eigenvalues of the expansions and rotations of powers of $A$.
Given the linear independence of the set $\{[A^e]_\bb,[A^r_{kl}]_\bb\}$ it is clear that this generic system is over-determined for $n\ge 3.$ Here it is for $n=2$:

\begin{align*}
&\ \tr({\bA^2}^e) -\tr(A)\tr(\bA^e)+2\det(A)=0,\\
&\ \tr({\bA^2}^r) -\tr(A)\tr(\bA^r)=0.
\end{align*}
Using the definitions of ${\bA^2}^e, {\bA^2}^r$ and the basis $\{\hat u,\hat u^\bot\},$ these identities can easily be shown to produce the results \eqref{invariants n=2}.

As an aside, we could hope to express all the coefficients in \eqref{tr_CH2}, \eqref{tr_CH4} in terms of the invariants of $\bA^e$ and $\bA^r_{kl}$ alone
(and not the higher powers) but this is unrealistic given the following recursions derived from \eqref{main result}:
\begin{align*}
& A^{m+1}(u)=A(A^m(u))=(\bA^m)^e(u)A(u)+\sum_{k<l}(\bA^m)^r_{kl}(u)A(R_{kl}(u))\\
\implies & (\bA^{m+1})^e(u)=(\bA^m)^e(u)\bA^e(u)+\sum_{k<l}(\bA^m)^r_{kl}(u)(\bA^T)^r_{kl}(u),
\end{align*}
and similarly,
\begin{equation}
(\bA^{m+1})^r_{pq}(u)=(\bA^m)^e(u)\bA^r_{pq}(u)+\sum_{k<l}(\bA^m)^r_{kl}(u)A(R_{kl}(u))\cdot R_{pq}(u).\label{recur_2}
\end{equation}
The terms $A(R_{kl}(u))\cdot R_{pq}(u)$ in \eqref{recur_2} look problematic but putting $u=b_p$, for example, gives the diagonal terms needed for the trace
\begin{align*}
(\bA^{m+1})^r_{pq}(b_p)=(\bA^m)^e(b_p)\bA^r_{pq}(b_p)&+(\bA^m)^r_{pq}(b_p)\bA^e(b_q)\\ 
&+\sum_{p<l<q}(\bA^m)^r_{pl}(b_p)\bA^r_{lq}(b_l)\\ 
&-\sum_{p<q<l}(\bA^m)^r_{pl}(b_p)\bA^r_{ql}(b_l)\\ 
&-\sum_{k<p<q}(\bA^m)^r_{kp}(b_p)\bA^r_{kq}(b_k).
\end{align*}
Nonetheless, it seems unlikely from this result that either $\tr((\bA^{m+1})^e)$ or $\tr((\bA^{m+1})^r_{pq})$ can be expressed in terms of the powers of
traces of lower order rotations and expansions.

When $A$ is normal, the situation is radically simplified by proposition \ref{normal II}.
\begin{cor}
When the endomorphism $A$ is normal the system \eqref{CH2},\eqref{CH4} has generic rank $n$ for the matrix invariants of $A$.
\end{cor}

\begin{proof}
Using a basis $\bb$ which simultaneously diagonalises $[A^e]_\bb,\dots,[(A^n)^e]_\bb$ equation \eqref{CH2} generates the system
\begin{equation}
\lambda^{ne}_i -\text{pm}^1(A)\lambda^{(n-1)e}_i+\dots+(-1)^{n-1}\text{pm}^{n-1}(A)\lambda^e_i+(-1)^n\text{pm}^n(A)=0,\ i=1,\dots,n \label{normal_III}
\end{equation}
where the $\lambda^{ke}_i$ are eigenvalues of $(\bA^k)^e$ and $\lambda^e_i$ are eigenvalues of $\bA^e.$ The rank of \eqref{normal_III} is equal to the
number of distinct eigenvalues of $\bA^e$, remembering that this is strictly less than $n$ because $A$ is normal and assumed not symmetric.\newline
The situation with \eqref{CH4} in the normal case is more complicated. Set $[A]_\bb=D+S$ with $DS=SD$ and recall that (normal or not) $(\bA^p)^r_{kl}(b_m)$ is
zero unless $m=k,l$, that $(\bA^p)^r_{kl}(b_k)=(\bA^p)^r_{kl}(b_l)=(A^{p^{\skew}})^l_k$ and that $S^2$ is diagonal; using $A^p=(D+S)^p$
so that $A^{p^{\skew}}=\onehalf((D+S)^p-(D-S)^p),$ it is straightforward to show that
\begin{align*}
(\bA^p)^r_{kl}(b_k)&=\sum_{\stackrel{m=0,}{p-m\ \text{odd}}}^{p-1}\binom{p}{p-m}(D^l_l)^m(S^{p-m-1})^l_lS^l_k\\
&=\sum_{\stackrel{m=0,}{p-m\ \text{odd}}}^{p-1}\binom{p}{p-m}(\lambda^e_l)^m((S^2)^l_l)^\frac{p-m-1}{2}S^l_k.
\end{align*}
Then putting $u=b_k$ in \eqref{CH4} gives
\begin{align}
0&=\sum_{p=1}^{n}(-1)^{n-p}\text{pm}^{n-p}(A)(\bA^p)^r_{kl}(b_k)\notag\\
&=S^l_k\sum_{p=1}^{n}\left((-1)^{n-p}\text{pm}^{n-p}(A)\left(\sum_{\stackrel{m=0,}{p-m\ \text{odd}}}^{p-1}\binom{p}{p-m}(\lambda^e_l)^m((S^2)^l_l)^\frac{p-m-1}{2}\right)\right),\label{normal_IV}
\end{align}
remembering that $S^l_k$ is only non-zero when $\lambda^e_l=\lambda^e_k.$ So, in this case where $A$ is normal, the system \eqref{normal_III},\eqref{normal_IV} has
combined generic rank $n$ for the $n$ matrix invariants $\text{pm}^k(A).$
\end{proof}

Now we return to the Newton trace formulae \eqref{Newton} applied to $[A]_\bb=[A^{\sym}]_\bb+[A^{\skew}]_\bb$ on $\E^n$. By way of exploration consider the $k=2$ formula
\begin{align*}
&\tr(A^2)-\pmr^1(A)\tr(A)+2\pmr^2(A)=0\\
\iff\ &\tr({A^{\sym}}^2)+\tr({A^{\skew}}^2)+2\tr(A^{\sym}A^{\skew})-\tr(A^{\sym})^2+2\pmr^2(A)=0\\
\iff\ &\tr({A^{\sym}}^2)+\tr({A^{\skew}}^2)-\tr(A^{\sym})^2+2\pmr^2(A)=0.
\end{align*}
Applying the same trace formula to $[A^{\sym}]_\bb$ and $[A^{\skew}]_\bb$ and removing the leading terms immediately above gives
\begin{equation*}
\pmr^2(A)=\pmr^2(A^{\sym})+\pmr^2(A^{\skew}).
\end{equation*}
But more importantly for our goal,
\begin{equation}
\pmr^2(A)=\pmr^2(A^e)+\oneqtr\sum_{k<l}\tr(A^r_{kl})^2,
\end{equation}
using $\tr({A^{\skew}}^2)=-\onehalf\sum_{kl}\tr(A^r_{kl})^2$ arising from \eqref{skew decomp 2}.
However, attempting to create a similar identity for $\pmr^3(A)$ is obstructed by the term $\tr(A^{\sym}{A^{\skew}}^2).$
This can be dealt with by using a basis which diagonalises $A^\sym$ ($[A]_\bb=D+S$) so that the product explicitly contains
the eigenvalues and traces of the quadratic forms, similarly for the higher order terms that appear in the identities for $\pmr^k(A)$.
But in general this is unproductive, see the appendix for some further details. The full set of identities for $\E^3$ are given in the examples section.

Finally as an aside, recall that $\tr(AA^*)$ is not an elementary symmetric polynomial in the entries of $[A]_\bb$ but is an
invariant under change of basis. Can it, and other power sum symmetric polynomials be expressed in terms of rotational and expansionary invariants?
To hint at the answer we give without proof the following identity on $\E^n.$
\[
n\tr(AA^*)=2\sum_{k<l}\tr({A^r_{kl}}^2)+\tr(A^e)^2=-4\sum_{k<l}\pmr^2(A^r_{kl})+2\sum_{k<l}\tr(A^r_{kl})^2+\tr(A^e)^2.
\]

\section{Examples and applications}\label{Section 3}

\subsection{$\E^3$}\label{Section 3.1}\
Relative to an arbitrary orthonormal basis $\bb$ the matrix representations of the expansionary and rotational quadratic forms of an endomorphism $A$ on $\E^3$ are
\begin{align*}
[A^e]_\bb = [A^{\sym}]_\bb &=\left( \begin{array}{ccc} A^1_1 & \onehalf(A^1_2+A^2_1) & \onehalf(A^1_3+A^3_1) \\ \\
\onehalf(A^1_2+A^2_1) & A^2_2 & \onehalf(A^2_3+A^3_2) \\ \\
\onehalf(A^1_3+A^3_1) & \onehalf(A^2_3+A^3_2) & A^3_3 \end{array}\right),\\ \\
[A^r_{12}]_\bb &= \left( \begin{array}{ccc} A^2_1 & \onehalf(A^2_2-A^1_1) & \onehalf A^2_3 \\ \\
\onehalf(A^2_2-A^1_1) & -A^1_2 & -\onehalf A^1_3 \\ \\
\onehalf A^2_3 & -\onehalf A^1_3 & 0 \end{array}\right),\\ \\
[A^r_{13}]_\bb &= \left( \begin{array}{ccc} A^3_1 & \onehalf A^3_2 & \onehalf(A^3_3-A^1_1) \\ \\
\onehalf A^3_2 & 0 & -\onehalf A^1_2 \\ \\
\onehalf(A^3_3-A^1_1) & -\onehalf A^1_2 & -A^1_3 \end{array}\right),\\ \\
[A^r_{23}]_\bb &= \left( \begin{array}{ccc} 0 & \onehalf A^3_1 & -\onehalf A^2_1 \\ \\
\onehalf A^3_1 & A^3_2 & \onehalf (A^3_3-A^2_2) \\ \\
-\onehalf A^2_1 & \onehalf(A^3_3-A^2_2) & -A^2_3 \end{array}\right).
\end{align*}

In a basis which diagonalises $A^e$:
\begin{align*}
[A^e]_\bb = [A^{\sym}]_\bb &=\left( \begin{array}{ccc} \lambda^e_1 & 0 & 0 \\ \\
0 & \lambda^e_2 & 0 \\ \\
0 & 0 & \lambda^e_3 \end{array}\right),\\ \\
[A^r_{12}]_\bb &= \left( \begin{array}{ccc} -A^1_2 & \onehalf(\lambda^e_2-\lambda^e_1) & \onehalf A^2_3 \\ \\
\onehalf(\lambda^e_2-\lambda^e_1) & -A^1_2 & -\onehalf A^1_3 \\ \\
\onehalf A^2_3 & -\onehalf A^1_3 & 0 \end{array}\right),\ \tr(A^r_{12})=-2A^1_2 \\ \\
[A^r_{13}]_\bb &= \left( \begin{array}{ccc} -A^1_3 & -\onehalf A^2_3 & \onehalf(\lambda^e_3-\lambda^e_1) \\ \\
-\onehalf A^2_3 & 0 & -\onehalf A^1_2 \\ \\
\onehalf(\lambda^e_3-\lambda^e_1) & -\onehalf A^1_2 & -A^1_3 \end{array}\right),\ \tr(A^r_{13})=-2A^1_3 \\ \\
[A^r_{23}]_\bb &= \left( \begin{array}{ccc} 0 & -\onehalf A^1_3 & \onehalf A^1_2 \\ \\
-\onehalf A^1_3 & -A^2_3 & \onehalf (\lambda^e_3-\lambda^e_2) \\ \\
\onehalf A^1_2 & \onehalf(\lambda^e_3-\lambda^e_2) & -A^2_3 \end{array}\right),\ \tr(A^r_{23})=-2A^2_3.
\end{align*}

We have $[A]_\bb=D+S$ with $D=\diag(\lambda^e_1,\lambda^e_2,\lambda^e_3)$ and
\[
 S=\onehalf(\tr(A^r_{12})[R_{12}]+\tr(A^r_{13})[R_{13}]+\tr(A^r_{23})[R_{23}])=-A^1_2[R_{12}]-A^1_3[R_{13}]-A^2_3[R_{23}].
\]
Starting with the two Newton trace formulae
\begin{align*}
&\ \tr(\mtt{B}^2)-\pmr^1(\mtt{B})\tr(\mtt{B})+2\pmr^2(\mtt{B})=0,\\
&\ \tr(\mtt{B}^3)-\pmr^1(\mtt{B})\tr(\mtt{B}^2)+\pmr^2(\mtt{B})\tr(\mtt{B})-3\pmr^3(\mtt{B})=0,
\end{align*}
we note first that
\[
\tr(S^2)+2\pmr^2(S)=0\ \iff\ \tr(S^2)=-2((A^1_2)^2+(A^1_3)^2+(A^2_3)^2)=-\onehalf\sum_{k<l}\tr(A^r_{kl})^2
\]
and
\[
\tr(D^2)-\tr(D)^2+2\pmr^2(D)=0.
\]
Also (by direct calculation)
\[
\tr(DS^2)=-\oneqtr\left[\lambda^e_1\left(\tr(A^r_{12})^2+\tr(A^r_{13})^2\right)+\lambda^e_2\left(\tr(A^r_{12})^2+\tr(A^r_{23})^2\right)+\lambda^e_3\left(\tr(A^r_{13})^2+\tr(A^r_{23})^2\right)\right].
\]
Then the two  trace formulae for $[A]$ give
\begin{align*}
\pmr^2(A)&=\onehalf(\tr(D)^2-\tr(D^2)-\tr(S^2))= \pmr^2(D)+\pmr^2(S)\ \text{and}\ \\
\det(A)&=\onethird\left(\tr(D^3)+\onehalf\tr(D)^3\right)+\tr(S^2D)-\onehalf\tr(D)(\tr(S^2)+\tr(D^2))
\end{align*}
with each term involving the eigenvalues of the expansionary form and traces of the rotational forms. The elaboration of  \eqref{tr_CH2} and \eqref{tr_CH4} are too complex to be informative.

The following canonical form on $\E^3$ serves as an illustrative, non-normal, example of the determination of the eigenspaces of $A$ from the common zeros of the rotational forms.
In an arbitrary orthonormal basis $\bb$ suppose that
\[
[A]_\bb = \left( \begin{array}{ccc} \lambda & 1 & 0 \\
0 & \lambda & 0 \\
0 & 0 & \mu \end{array}\right),
\]
with $\lambda>\mu.$ Then
\[
[A^e]_\bb =\left( \begin{array}{ccc} \lambda & \onehalf & 0 \\
\onehalf & \lambda & 0 \\
0 & 0 & \mu \end{array}\right),
\]
and $[A^e]_\bb$ is diagonalised in the orthonormal basis
\[
\bb'_1=\frac{1}{\sqrt{2}}(\bb_1+\bb_2),\ \bb'_2=\frac{1}{\sqrt{2}}(\bb_1-\bb_2),\ \bb'_3=\bb_3
\]
with

\[
[A^e]_{\bb'} =\left( \begin{array}{ccc} \lambda +\onehalf & 0 & 0 \\
0 & \lambda-\onehalf & 0 \\
0 & 0 & \mu \end{array}\right),
\]
and
\[
[A]_{\bb'} = \left( \begin{array}{ccc} \lambda+\onehalf & -\onehalf & 0 \\
\onehalf & \lambda-\onehalf & 0 \\
0 & 0 & \mu \end{array}\right).
\]
The rotational quadratic forms have representations and eigenvalues
\begin{align*}
[A^r_{12}]_{\bb'} &= \left( \begin{array}{ccc} \onehalf & -\onehalf & 0 \\
-\onehalf & \onehalf & 0 \\
0 & 0 & 0 \end{array}\right),\ \{0,1\}\\
[A^r_{13}]_{\bb'} &= \left( \begin{array}{ccc} 0 & 0 & \onehalf(\mu-\lambda-\onehalf) \\
0 & 0 & 0 \\
\onehalf(\mu-\lambda-\onehalf) & 0 & 0 \end{array}\right),\ \{\ 0,\pm |\onehalf(\mu-\lambda-\onehalf)|\ \}\\
[A^r_{23}]_{\bb'} &= \left( \begin{array}{ccc} 0 & 0 & -\frac{1}{4} \\
0 & 0 & \onehalf (\mu-\lambda + \onehalf) \\
-\frac{1}{4} & \onehalf (\mu-\lambda + \onehalf) & 0 \end{array}\right),\ \{\ 0,\pm\onehalf\sqrt{\oneqtr+(\mu-\lambda+\onehalf)^2}\ \}.
\end{align*}
Clearly $\Sp\{\bb'_3\}$ is a common zero of the rotational forms and the only other common zero is $\Sp\{\bb'_1+\bb'_2\}$;
however their direct sum is not a common zero (see proposition \ref{Q forms1}) and so they belong to distinct eigenvalues of $A$,
each with geometric  multiplicity one (which, of course, can be directly verified from the original canonical form).
The corresponding eigenvalues of $A$ can be read directly from $[A^e]_\bb$ because of the choice of basis.

Notice also that the eigenvalues of $A$ lie between the maximum and minimum eigenvalues of $\bA^e$ and that
\[
[A^{\skew}]_{\bb'}=\onehalf\tr[A^r_{12}]_{\bb'}[R_{12}]
\]
with eigenvalues $0,\pm\onehalf i$ allowing both parts of  theorem \ref{Bromwich 06} to be verified.

\subsection{The Frenet shape map}\label{Section 3.2}\

In what follows we will examine  the {\em shape map} of a vector field -- an  endomorphism of tangent spaces to $\E^3, $ using the orthonormal Frenet frame
as the basis $\mathbf{b}$ of previous sections. This serves as a model for applications of the main theorem  \ref{thm1} to the geometry of connections.

Suppose that $T$ is a unit flow field, smooth and non-singular on some open subset $U$ of $\E^3$ and whose integral curves have positive curvature.
We will use natural coordinates $(x^a)$ and the flat connection $\nabla$ on $\E^3$ with
\[
\nabla_XY=X^a\pd{Y^b}{x^a}\pd{}{x^b}
\]
for vector fields $X=X^a\pd{}{x^a},\ Y=Y^a\pd{}{x^a}.$ At each point in $U$ we can define the Frenet frame, $\{T,N,B\},$ in the usual way for
the unique, arc length-parametrised integral curve of $T$ through that point (see eg \cite{St83}). The smoothness assumptions on the flow mean that
the Frenet formulae can be defined on all of $U$  with curvature and torsion, $\kappa, \tau$:

\begin{equation}\label{E:FFF}\begin{aligned}
\nabla_T T=& &\kappa N, &\\
\nabla_T N=& -\kappa T & & +\tau B,\\
\nabla_T B=& &-\tau N. &
\end{aligned}\end{equation}

With any linear connection $\hnabla$, with torsion $\hat T$ (not to be confused with $T$ or $\tau$), on a manifold $M$ there is a corresponding $(1,1)$ tensor,
$A_X:=\hnabla X+\hat T(X,\cdot),$ called the {\em shape map} (see \cite{JP01} and Kobayashi and Nomizu, Volume 1, p 235 \cite{KN63})
associated to any vector field X (on each tangent space $T_pM,$ $A_X$ acts as a vector space endomorphism). In our case $\hat T=0$ and the following hold
\begin{align*}
A_X(Y)=&\nabla_YX,\\
[X,Y]=&A_Y(X)-A_X(Y).
\end{align*}

Using these the Frenet formulae can be rephrased as
\begin{equation}\label{E:SMF}\begin{aligned}
A_T(T)=& &\kappa N&,\\
A_T(N)=& -\kappa T & & +\tau B &- [T,N]&,\\
A_T(B)=& &-\tau N & &- [T,B]&.
\end{aligned}\end{equation}
This produces the matrix representation of the shape map $A_T$ relative to the orthonormal Frenet basis $F:=\{T,N,B\}.$

\[
[A_T]_F = \left( \begin{array}{ccc} 0 & -\kappa & 0 \\
\kappa & 0 & -\tau+\sigma \\
0&\tau-\sigma&0 \end{array}\right).
\]

This matrix representation is skew symmetric because  $[T,N]\in\Sp\{B\}$, $[T,B]\in\Sp\{N\}$ and  $\sigma:=[T,N]\cdot B=-[T,B]\cdot N;$
this includes the case $[T,N]=0$ where the osculating planes (at each point these are spanned by $T$ and $N$) are integrable (see also \cite{Costa13}).

We can now compute the expansion and rotations induced by the flow through the action of $A_T$. The quadratic form $\bA^e_T=0$ because of the skew symmetry just noted,
and is to be expected since $T$ is unit.
Denoting the rotational quadratic forms of $A_T$ by, for example, $\bA^r_{(TN)},$ in the case of the osculating 2-planes, $\Sp\{T,N\},$ we obtain

\[
[A^r_{(TN)}]_F = \left( \begin{array}{ccc} \kappa & 0 & \frac{1}{2}(\sigma-\tau) \\
0 & \kappa & 0 \\
\frac{1}{2}(\sigma-\tau)&0&0 \end{array}\right),
\]

\[
[A^r_{(TB)}]_F = \left( \begin{array}{ccc} 0 & \frac{1}{2}(\tau-\sigma) & 0 \\
\frac{1}{2}(\tau-\sigma) & 0 & \frac{1}{2}\kappa \\
0&\frac{1}{2}\kappa&0 \end{array}\right),
\]

\[
[A^r_{(NB)}]_F = \left( \begin{array}{ccc} 0 & 0 & -\frac{1}{2}\kappa \\
0 & \tau-\sigma & 0 \\
-\frac{1}{2}\kappa&0&\tau-\sigma \end{array}\right).
\]
Observe that the curvature $\kappa$ is the necessarily common rotation of $T$ and $N$ in the osculating plane at each point
and that $\tau-\sigma$ is the rotation at each point of $N$ and $B$ in the normal plane. As expected the three rotational quadratic forms have a single common zero,
namely $(\sigma-\tau)T-\kappa B$, the generator of the kernel of $A_T.$ In the case of a single curve only the torsion accounts for the rotation (twist) of the normal plane,
so flow design with $\tau=\sigma$ may be of further interest (see, for example, \cite{W08}).

\section{Discussion and open questions}
The geometric thrust of this paper is the novel decomposition of an endomorphism into its expansionary and rotational parts using quadratic forms.
This allows the identification of eigendirections as the common  zeros of the rotational forms with the geometric multiplicity being given by the dimensions of those zero-valued subspaces.
The eigenvalues are produced by evaluation of the corresponding expansions. The relationship between the eigenvalues of the endomorphism and
those of its geometric quadratic forms is less clear cut and probably less useful. The applications to differential geometry and differential equations will almost
certainly involve evolution equations for the quadratic form invariants.

Some open questions have occurred to me, either because I have been unable to answer them or because they have been peripheral to task at hand. Here are a few.

The quasi rotations give a natural decomposition of the orthogonal complement of a one dimensional subspace. Are there useful analogues for higher dimensional subspaces?

While the matrix invariants of an endomorphism can be constructed from the corresponding real quadratic form invariants,
apart from the two dimensional case rotations seem to be represented by the sum of their traces alone -- are there roles for the higher invariants of the rotations?
And what can be said about the signatures of the rotational forms defined by different orthonormal bases?

What also can be said about endomorphisms whose average expansion is attained along an eigendirection? Or whose average rotations are zero?
How can 2-planes of extremal rotation be identified, independent of basis choice?
Finally, how useful would it be to build a classification scheme for endomorphisms characterised by the signatures of the quadratic forms,
assuming that these can be well-defined independent of the choice of basis.

\begin{appendix}
\section{Appendix}
\subsection{Some definitions and identities}
For the sake of brevity in what follows $R_{ij}$ is used for both the quasi-rotations as maps and as matrices; it is implicit that $i<j$ in $R_{ij}.$
$\delta^{i_1\dots i_m}_{j_1\dots j_m}$ is the generalised Kronecker delta and $\varepsilon_{i_1\dots i_m}$ is the Levi-Civita symbol; the epsilon-delta identity is
\[
\varepsilon_{i_1\dots i_ki_{k+1}\dots i_n}^{i_1\dots i_kj_{k+1}\dots j_n}=k!\ \delta^{j_{k+1}\dots j_n}_{i_{k+1}\dots i_n}.
\]

\begin{eqnarray*}
(R_{pq})^i_j &=-\delta^p_i\delta^q_j+\delta^q_i\delta^p_j &= -\delta^{pq}_{ij},\\
{(R^2_{pq})}^i_j &=(R_{pq})^i_k(R_{pq})^k_j &=\sum_k\delta^{pq}_{ik}\delta^{pq}_{kj}.
\end{eqnarray*}

For non-zero $u$ and unit $v \in \E^n:$
\begin{eqnarray*}
u\cdot R_{pq}(u) =0,&\ u=\frac{-1}{n-1}\displaystyle{\sum_{r<s}}R^2_{rs}(u),& \\ \\
(R_{pq}(u))^i =\delta^i_qu^p-\delta^i_pu^q,&\  ({R^2_{pq}}(u))^i =-\delta^i_pu^p-\delta^i_qu^q\ \text{(no sum)}.&
\end{eqnarray*}

From proposition \ref{magic propn},
\begin{align*}
u&=(u\cdot v)v+\sum_{k<l}(u\cdot v^\bot_{kl})v^\bot_{kl},\ v^\bot_{kl}:=R_{kl}(v),\\
\|u\|^2&=(u\cdot v)^2+\sum_{k<l}(u\cdot v^\bot_{kl})^2.
\end{align*}

For an endomorphism $A$ on $\E^n:$
\[
\bA^e(v):=A(v)\cdot v,\ \bA^r_{kl}(v):= A(v)\cdot R_{kl}(v).
\]
From theorem \ref{main result} and corollary \ref{main thm cor},
\begin{align*}
A(u)&=\bA^e(\hat u)u+\sum_{k<l}\bA^r_{kl}(\hat u)R_{kl}(u),\\
\|A(\hat u)\|^2&=\mathbf{A}^e(\hat u)^2+\sum_{k<l}{\mathbf{A}^r_{kl}(\hat u)}^2.
\end{align*}

From proposition \ref{cor comms},
\begin{align*}
A^{\sym}(u)&=\bA^e(\hat u)u+\onehalf\sum_{k<l}\left([A,R_{kl}](\hat u)\cdot\hat u\right) R_{kl}(u),\\
A^{\skew}(u)&=-\onehalf\sum_{k<l}\left(\{A,R_{kl}\}(\hat u)\cdot \hat u\right) R_{kl}(u).
\end{align*}

From proposition \ref{cor comms II},
\begin{align*}
\bA^e(\hat u)&=A^{\sym}(\hat u)\cdot\hat u,\\
\bA^r_{kl}(\hat u)&=\onehalf[A^{\sym},R_{kl}](\hat u)\cdot \hat u-\onehalf\{A^{\skew},R_{kl}\}(\hat u)\cdot \hat u.
\end{align*}

\subsection{Remarks on the identities arising from Newton trace formulae}
Consider the case where the matrix $\mtt{B}$ decomposes additively into  nontrivial symmetric and skew-symmetric parts, and furthermore assume that the symmetric part is diagonal
since this can be achieved by an orthogonal transformation, that is $\mtt{B}=\mtt{D}+\mtt{S}$ with $\mtt{D}$ diagonal and $\mtt{S}$ skew-symmetric. So consider $\text{tr}((\mtt{D}+\mtt{S})^k)$:
\begin{align}
\text{tr}((\mtt{D}+\mtt{S})^k)=\text{tr}(\mtt{D}^k &+(\text{terms with $k-1$ occurrences of $\mtt{D}$})\label{trace expn}\\
&+\dots+(\text{terms with $1$ occurrence of $\mtt{D}$})+\mtt{S}^k).\notag
\end{align}
Using the property $\text{tr}(\mtt{B^T})=\text{tr}(\mtt{B}),$ the traces of terms in this expansion with an odd occurrence of $\mtt{S}$ can be shown to sum to zero as follows:
for each term $\mathtt{M_1M_2\dots M_k}$
the reverse, $\mathtt{M_k\dots M_1}$, also appears exactly once ($\mtt{M_i}$ is either $\mtt{D}$ or $\mtt{S}$) unless it is equal to its reverse.
Now $\text{tr}(\mtt{M_k}\dots \mtt{M_1})=\text{tr}((\mtt{M_k}\dots \mtt{M_1})^\mtt{T}) =\text{tr}(\mtt{M^T_1}\dots \mtt{M^T_k})
= -\text{tr}(\mtt{M_1}\dots \mtt{M_k}) $
since $\mtt{S}$ appears an odd number of times. So in this case  $\tr(\mtt{M_1}\dots \mtt{M_k})  +\tr(\mtt{M_k}\dots \mtt{M_1})=0$
removing all terms with an odd occurrence of $\mtt{S};$ this trivially includes self-reverse case.\newline
Of those terms with a fixed even occurrence of $\mtt{S}$ the traces are the same for all cyclic permutations and all reversals of each term $\mtt{M_1}\dots \mtt{M_k}$
because of the properties $\tr(\mtt{B}_1\mtt{B}_2)=\tr(\mtt{B}_2\mtt{B}_1)$ and $\tr(\mtt{B^T})=\tr(\mtt{B})$
but otherwise different. For example, in general $\tr(\mathtt{DSDS})~\neq~\tr(\mathtt{DDSS}).$ They are equal when $\mtt{B}$ is normal. \newline

Bearing these properties in mind, the expansion \eqref{trace expn} for $k=4$ is, for example,
\begin{align*}
\tr(\mtt{B}^4)=\tr (\mtt{(D+S)^4})&= \tr (\mtt{D^4+S^4})+\tr (\mtt{DDDS+DDSD+DSDD+SDDD})\\
&+ \tr (\mtt{DDSS+DSDS+DSSD+SDSD+SSDD+SDDS})\\
&+ \tr (\mtt{DSSS+SDSS+SSDS+SSSD})\\
&=\tr (\mtt{D^4})+\tr (\mtt{S^4})+4\tr  (\mtt{DDSS}) + 2\tr(\mtt{SDSD}).
\end{align*}
When $\mtt{B}$ is normal
\[
\tr(\mtt{B}^4)=\tr (\mtt{D^4})+\tr (\mtt{S^4})+6\tr(\mtt{(DS)^2}).
\]

Now we are in a position to use the Newton trace formulae to express the matrix invariants of our endomorphism $A$ in terms of the traces of
powers and products of $D$ and $S$ from \eqref{MC1},\eqref{MC2} which,
at least for traces of powers of $D$ and $S,$ can in turn be written in terms of their matrix invariants via the trace formula.\newline
For example, for $n=4$ we have,
\begin{equation*}
\det(A)=\det(D)+\det(S)-\tr(DDSS)-\onehalf\tr(SDSD)+\tr(SDS)\tr(D)+\text{pm}^2(D)\text{pm}^2(S).
\end{equation*}
And, for comparison with \eqref{invariants n=2} when $n=2$,
\[
\det(A)=\det(D)-\onehalf\tr(S^2).
\]
Of course, the right hand side could  be expressed entirely in terms of traces by using the Newton formulae recursively.
However, as pointed out earlier terms like $\tr(SDSD)$ can't be expressed directly in terms of the invariants of the rotational and expansionary quadratic forms
unless the basis diagonalising $[A^e]$ is used.

\subsection{Diagonal expansion of determinants}

Although the little-known result below \cite{Coll83} has not explicitly been cited so far I have used it extensively during the development of those
results presented here which use the decomposition $[A]_\bb=D+S,$ so that $B=S$ in what follows.\newline
Let $\cS:=\{1,\dots,m\},$ $\theta\subset \cS$ and $\bar{\theta}$ the complement of $\theta$ in $\cS.$ If $\mtt{B}$ is $m\times m$ then $\mtt{B}(\theta)$
is the matrix obtained from $\mtt{B}$ by removing the rows and columns not indexed in $\cS$ by $\theta.$ We take $\mtt{B}(\emptyset):=\mtt{I_1}.$

\begin{propn} (Collings 1983)
Let $\mtt{D}$ be an $m\times m$ diagonal matrix and let $\mtt{B}$ be an arbitrary $m\times m$ matrix. Then

\[
\det(\mtt{D}+\mtt{B})=\sum_{\theta\subset\cS}\det(\mtt{D}(\bar\theta)).\det(\mtt{B}(\theta)).
\]
\end{propn}

\end{appendix}

\section*{Acknowledgements}
Thanks particularly to Peter Stacey and also to Thoan Do, Khanh Pham and Peter Forrester for helpful discussions over many years.

\end{document}